\documentclass[english]{amsart}

\usepackage{amsmath,amssymb,enumerate, bbm}

\usepackage[T1]{fontenc}
\usepackage[utf8]{inputenc}
\usepackage[all]{xy}

\usepackage{babel}
\usepackage{amstext}
\usepackage{amsmath}
\usepackage{amsfonts}
\usepackage{latexsym}
\usepackage{ifthen}

\usepackage{xypic}
\xyoption{all}
\newtheorem{lemma1}{}[section]

\newenvironment{theorem}{\begin{lemma1}{\bf Theorem.}}{\end{lemma1}}
\newenvironment{proposition}{\begin{lemma1}{\bf Proposition.}}{\end{lemma1}}
\newenvironment{corollary}{\begin{lemma1}{\bf Corollary.}}{\end{lemma1}}

\newenvironment{conjecture}{\begin {lemma1}{\bf Conjecture.}}{\end{lemma1}}

\newenvironment{remark*}{{\bf Remark.}}{}
\newenvironment{example*}{{\bf Example.}}{}

\newcommand{\R}{\ensuremath{\mathbb{R}}}
\newcommand{\Q}{\ensuremath{\mathbb{Q}}}
\newcommand{\Z}{\ensuremath{\mathbb{Z}}}

\newcommand{\holom}[3]{\ensuremath{#1:#2  \rightarrow #3}}

\makeatletter
\ifnum\@ptsize=0  \fi \ifnum\@ptsize=2  \fi 

\newcommand\sO{{\mathcal O}}

\title{Anticanonical system of Fano fivefolds} 
\date{\today}
\author{Andreas H\"oring \and Robert \'Smiech}

\address{Andreas H\"oring, Universit\'e C\^ote d'Azur, CNRS, LJAD, France}
\email{Andreas.Hoering@unice.fr}

\address{Robert \'Smiech, University of Warsaw, Faculty of Mathematics, Informatics and Mechanics, ul. Banacha 2, 02-097 Warszawa}
\email{r.smiech@mimuw.edu.pl}

\begin{document}

\begin{abstract} 
	We show that any Fano fivefold with canonical Gorenstein singularities has an effective anticanonical divisor. Moreover, if a general element of the anticanonical system is reduced, then it has canonical singularities.
\end{abstract}

\maketitle

%\tableofcontents

\vspace{-1ex}

\section{Introduction}

The anticanonical linear system is a fundamental tool in the classification theory of Fano varieties, but unfortunately not much is known in higher dimension. In this note we clarify the situation in dimension five:

\begin{theorem}  \label{theoremmain}
	Let $X$ be a normal projective variety of dimension five with canonical Gorenstein singularities.
	Suppose that $X$ is Fano, i.e. the anticanonical divisor $-K_X$ is
	ample.
	
	\begin{enumerate}
		\item Then one has $h^0(X,-K_X) \ge 4$.
		\item Assume that a general element of $D \in |-K_X|$ is reduced, then $D$ has canonical singularities.
	\end{enumerate}
	
\end{theorem}

We expect that Fano varieties with canonical singularities always admit
an effective anticanonical divisor with canonical singularities.
In fact the arguments in this paper show that this would be a consequence
of the effective nonvanishing conjecture, attributed to Kawamata and Ambro:

\begin{conjecture} \cite[Conj. 2.1]{Kaw00} \label{conjecturenonv}
	Let $X$ be a complete normal variety, and $\Delta$ an effective $\R$-divisor on $X$ such that $(X,\Delta)$ is klt. Let $L$ be a nef Cartier divisor on $X$ such that $L -(K_X+\Delta)$ is nef and big. Then $H^0(X,L)\ne 0$.
\end{conjecture}

Recall that the index $i_X$ of Fano variety is the biggest integer such that in the Picard group we have an equality $-K_X = i_X L$ for some ample Cartier divisor $L$. The divisor $L$ is called the fundamental divisor. Our theorem has the following consequence:

\begin{corollary} \label{corollarymain}
	Let $X$ be a normal projective variety of dimension five with canonical Gorenstein singularities.
	Suppose that $X$ is Fano, i.e. the anticanonical divisor $-K_X$
	ample. Then the fundamental system $|L|$ is not empty, in particular
	the anticanonical system $|-K_X|$ is not empty.
\end{corollary}

Indeed by Fujita (\cite{Fuj75}, \cite[Ch.I, Thm 5.10]{Fuj90}) 
	a Fano variety of dimension $n$ and index $i_X$ at least $n$ is a projective
	space, quadric, scroll over $\mathbb{P}^1$ or a Veronese surface, so the statement is trivial in this case.
The case of index $i_X \geq n-2$ was solved by Iskovskikh and Fujita (see \cite[Cor. 2.1.14]{IP99}, and Mella's paper for the properties of the fundamental divisor \cite{Me}). 
In the case of index $i_X=n-3$ a classification is currently out of reach, but Floris \cite[Prop. 3.2]{Flo13} showed nonvanishing in dimension five.
If $X$ is smooth and $i_X=n-3$, the combination of \cite{Liu17} and \cite{Flo13} yields the existence of effective fundamental divisors in any dimension. Thus in dimension five the last open case was $i_X=1$
which is covered by our theorem.

If $X$ has log-canonical singularities, we obtain a non-vanishing result:

\begin{proposition}  \label{propositionmain}
Let $X$ be a normal projective variety of dimension at most five with log-canonical Gorenstein singularities.
Suppose that $X$ is Fano, i.e. the anticanonical divisor $-K_X$ is
ample. Then one has $H^0(X,-K_X) \neq  0$.
\end{proposition}

Since log-canonical singularities may not be rational,
the Riemann-Roch computation from the proof of Theorem \ref{theoremmain} can not work in this setting. An extension theorem of Fujino allows us to reduce the problem
to a klt pair of lower dimension.

\section{Basic facts}
We work over the complex numbers. For definitions and basic facts around the minimal model program we refer to \cite{KM98}. 

For brevity we will identify line bundles with their first Chern class, i.e. we will write $L^k$ instead of $c_1(L)^k$. We may observe that we have $c_1(T_X) = -K_X = i_X L$.

Let $X$ be a normal projective variety of dimension $n$ with at most klt singularities.
Since klt singularities are rational, we know that for any Cartier divisor $L$ on 
$X$ and any (partial) resolution of singularities $\holom{\mu}{X'}{X}$, one has
\[
h^j(X', \mu^* L) = h^j(X, L) \qquad \forall \ j \in \{ 0, \ldots, \dim X \}.
\]
In particular we can compute the Hilbert polynomial $\chi(X,tL)$ on some resolution via the Riemann-Roch formula. By the projection formula we obtain
\begin{equation}\label{eq:HRR1}
\chi(X,tL) = \frac{L^n}{n!}t^n + \frac{-K_X\cdot L^{n-1}}{2\cdot (n-1)!}t^{n-1} +r(t)
\end{equation}
where $r(t)$ is some degree $n-2$ polynomial, with $r(0)=\chi(\sO_X)$.

If $X$ is smooth in codimension two, then the second Chern class $c_2(X)$ is well-defined and we have (e.g \cite[Sect. 2]{Hoe12})

\begin{equation}\label{eq:HRR2}
\chi(X,tL) = \frac{L^n}{n!}t^n + \frac{-K_X\cdot L^{n-1}}{2\cdot (n-1)!}t^{n-1} + \frac{((-K_X)^2 +c_2(TX)) \cdot L^{n-2}}{12\cdot (n-2)!}t^{n-2} +\tilde{r}(t)
\end{equation}
where $\tilde{r}(t)$ is some degree $n-3$ polynomial, with $\tilde{r}(0)=\chi(\sO_X)$.

We prove a special case of Conjecture \ref{conjecturenonv} in dimension three.

\begin{proposition} \label{propositionklt}
Let $X$ be a normal projective threefold such that $-K_X$ is generically nef, i.e. for every nef Cartier divisor $H$ on $X$ one has 
\[
-K_X \cdot H^2 \geq 0.
\]
Suppose that there exists an effective $\R$-divisor $\Delta$
such that $(X, \Delta)$ is klt.
Let $D$ be a nef Cartier divisor such that $D-(K_X+\Delta)$ is nef and big. Then we have
\[
H^0(X, D) \neq 0.
\]
\end{proposition}

\begin{proof}
By  \cite[Thm 3.1]{Kaw00} we can assume that $D$ is nef and big.
By \cite[Cor. 1.37]{Kol13} we know that $X$ has a $\Q$-factorial modification,
by the projection formula we can replace $X$ without loss of generality
by a $\Q$-factorial modification.

Since $-K_X$ is generically nef, we know by 
\cite[Thm.3.10]{Deb01} that the variety $X$ is either uniruled or $K_X$ is numerically trivial. If $K_X$ is numerically trivial and $X$ has canonical singularities, then $D-K_X \equiv D$ is nef and big and we conclude  by \cite[Prop. 4.1]{Kaw00}.
If $K_X$ is numerically trivial and $X$ does not have canonical singularities, the canonical divisor $K_{X'}$ of a canonical modification $X' \rightarrow X$ \cite[Thm 1.31]{Kol13} is numerically equivalent to a non-zero anti-effective $\Q$-divisor. In particular $X'$ (and hence $X$) is uniruled by \cite[Thm 3.10]{Deb01}.

Suppose now that $X$ is uniruled. Since $X$ has rational singularities we have  $h^i(X, \sO_X) = h^i(X', \sO_{X'})$ for any $i \in \{ 0, \ldots, 3 \}$ and any resolution of singularities $X' \rightarrow X$. Since $X'$ is uniruled, we obtain
$h^3(X, \sO_X) = h^3(X', \sO_{X'})=0$. 
By \cite[Cor. 4]{Xie09} we can also suppose without loss of generality that $h^1(X, \sO_X)=0$. Thus we can suppose that $\chi(X, \sO_X) \geq 1$.

Since $D-(K_X+\Delta)$ is nef and big, we have 
\[
h^0(X, D)=\chi(X, D)
\]
by Kawamata-Viehweg vanishing.  
Consider  the Hilbert polynomial
\[
\chi(X, tD) = \frac{D^3}{3!} (t^3+at^2+bt+c).
\]
Since $D$ is nef and big and $X$ has rational singularities, we have 
$h^j(X, -D)=0$ for all $0 \leq j \leq 2$.
Thus we have $\chi(X, -D)=-h^3(X, -D) \leq 0$ and
\[
\chi(X, D) \geq \chi(X, D) + \chi(X, -D) = \frac{D^3}{3} (a+c).
\]
Since $D^3>0$ we are left to show that $a+c>0$.

Comparing coefficients with \eqref{eq:HRR1} we get 
\[
a \frac{D^3}{3!}= \frac{-K_X \cdot D^2}{4}
\qquad \mbox{and} \qquad 
c \frac{D^3}{3!}=\chi(X, \sO_X).
\] 
By assumption $-K_X \cdot D^2 \geq 0$, so we have $a \geq 0$.
Since $\chi(X, \sO_X)>0$ we have $c>0$. Thus we finally obtain
$\chi(X, D) \geq \frac{D^3}{3} (a+c) > 0.$
\end{proof}

\begin{corollary} \label{corollarylogCY}
Let $X$ be a normal projective threefold. 
Suppose that there exists an effective $\R$-divisor $\Delta$
such that $(X, \Delta)$ is klt and $-(K_X+\Delta)$ is pseudoeffective.
Let $D$ be a nef and big Cartier divisor. Then we have
\[
H^0(X, D) \neq 0.
\]
\end{corollary}

\begin{proof}
By assumption $-K_X$ is numerically equivalent to the sum of a pseudoeffective and an effective divisor, so it is generically nef. Conclude with Proposition \ref{propositionklt}.
\end{proof}

\section{Proof of the main theorem}

\begin{proof}[Proof of the first statement of Theorem \ref{theoremmain}.]
By \cite[Cor 1.4.4]{BCHM06} we have a birational map $f: X' \rightarrow X$ such that $f^* (K_X) = K_{X'}$ and $X'$ has terminal Gorenstein singularities.
Since $h^0(X, -K_X)=h^0(X', -K_{X'})$ we will assume from now that $X$ is terminal and $-K_X$ is nef and big.

If $L$ is a nef and big divisor on the terminal variety $X$ we have, by Kawamata-Viehweg vanishing, that $H^p (X, L) = 0$ for $p >0$. 
In particular one has $\chi (X, L) = h^0 (X,L)$.

Suppose now that $-K_X=i_X L$, so $L$ is a fundamental divisor on $X$.
By Serre duality the Hilbert polynomial $p(t) = \chi(X, tL)$ satisfies $p(t)= -p(-t-i_X)$. We know that $p(t)$ is given by \eqref{eq:HRR2}, so by comparing coefficients of $p(t)$ and $-p(-t-i_X)$ we obtain that:
\begin{equation} \label{equationr}
\tilde{r}(t) = \frac{i_X \cdot L^3 \cdot c_2(TX)}{2\cdot 4!}t^2 + \left(\frac{2}{i_X}+ \frac{i_X^2 \cdot L^3 \cdot c_2(TX)}{6\cdot 4!} - \frac{i_X^4 \cdot L^5}{ 6!}\right)t + 1.
\end{equation}
Therefore we have obtained the precise formula for the Hilbert polynomial of a terminal Gorenstein weak Fano fivefold.  We are ready to conclude by evaluating the obtained polynomial for $t=i_X$:
\begin{equation} \label{eq:HRR3}
\chi(X,-K_X) = p(i_X)= \frac{(-K_X)^5}{4!} + \frac{c_2(TX) \cdot (-K_X)^3}{4!} + 3
\end{equation}
and since $(-K_X)^5>0$ and $c_2(TX) (-K_X)^3 \ge 0$ by \cite[Cor. 0.5]{Ou17} we obtain $h^0(X,-K_X) \ge 4$.
\end{proof}

\begin{remark*}
It is interesting to observe that \eqref{equationr}  
does not imply the nonvanishing for fundamental divisor in the case $i_X \ge 4$. For all values of $i_X$ we have:
\begin{equation} \label{eq:evaluation}
p(1) = \frac{L^5}{6!} \left( 6 + 15i_X +10i_X^2 - i_X^4 \right) + \frac{c_2(TX)\cdot L^3}{6 \cdot 4!} \left( 2 + 3i_X + i_X^2 \right) + \frac{2}{i_X} + 1,
\end{equation}
but the first term is negative for $i_X \ge 4$. However if we restrict ourselves to smooth fivefolds with $b_2 =1$ then the tangent bundle is stable \cite[Thm 2]{Hw} and we may use Bogomolov-Gieseker inequality which in our case has the form:
\begin{displaymath}
10 c_2(TX)\cdot L^3 \ge  4 c_1(TX)^2 \cdot L^3
\end{displaymath}
to obtain:
\begin{equation}
p(1) \ge \frac{L^5}{6!} \left( 6 +15i_X +14i_X^2 + 6i_X^3 + i_X^4 \right) + \frac{2}{i_X} + 1 \ge 2
\end{equation}
This interplay between stability and effective nonvanishing also appears in \cite{Flo13, Liu17}.            \end{remark*}

\begin{proof}[Proof of the second statement in Theorem \ref{theoremmain}.] 
We follow arguments from \cite[Sect. 5]{Heu16}.
By inversion of adjunction the statement is equivalent to proving that the pair $(X, D)$ is plt.
Arguing by contradiction we suppose that there exists a $0< c \leq 1$ such that the pair $(X, cD)$ 
is properly lc and, in the case $c=1$, not plt. 
By \cite[Lemma 5.1]{Amb99} there exists
a minimal lc centre  $W \subset X$ for the pair $(X, cD)$
that is contained in the base locus of $-K_X$. Thus the restriction map
\[
H^0(X, -K_X) \rightarrow H^0(W, -K_X|_W)
\]
is zero. On the other hand the divisor class
\[
-K_X - (K_X+cD) = (2-c) (-K_X)
\]
is ample, so by \cite[Thm 2.2]{Fuj11b} the restriction map is surjective. Thus we obtain that
\[
H^0(W, -K_X|_W) = 0.
\]
Since $W$ is a minimal lc centre, there exists by Kawamata subadjunction an effective $\Q$-divisor $\Delta_W$ such that
$(W, \Delta_W)$ is klt and
\[
K_W + \Delta_W \sim_\Q (K_X+cD)|_W \sim_\Q (1-c) K_X|_W.
\]
Since $c \leq 1$ we see that $(W, \Delta_W)$ is log Fano or log Calabi-Yau. 

Now we use our assumption: since $D$ is reduced, none of its irreducible components is a minimal lc centre of $(X, cD)$ with $c \leq 1$. Thus $W$ has dimension at most three. 
If $\dim W=2$ we  have $H^0(W, -K_X|_W) \neq 0$ by \cite[Thm 3.1]{Kaw00}, a contradiction. If $\dim W=3$, we obtain the same contradiction
by Corollary \ref{corollarylogCY} (see also
\cite[Thm 5.1]{Kaw00} for the log Fano case).
\end{proof}

\begin{remark*}
In the setting of Theorem \ref{theoremmain} the assumption that $D$ is reduced is equivalent to assuming that the fixed part of the anticanonical system 
is reduced. If $X$ is factorial with $\rho(X)=1$ the fixed part is always empty:
since $h^0(X, -K_X)>1$ the case $\mbox{Pic}(X) \simeq \Z K_X$ is elementary, 
if $i(X) \geq 2$ we apply \cite[Thm 1.1]{Flo13}.
\end{remark*}

\begin{proof}[Proof of the Proposition \ref{propositionmain}.] 
If $X$ has canonical singularities we conclude
with Theorem \ref{theoremmain} for $\dim X = 5$, \cite[Thm 5.2]{Kaw00} for $\dim X = 4$ and \cite[Thm 5.1]{Kaw00} for $\dim X \le 3$.
Thus, since $X$ is Gorenstein, we can suppose without
loss of generality that the non-klt locus is not empty.
In particular there exists a minimal lc centre $W$
for the pair $(X, 0)$.
The divisor class
\[
-K_X - (K_X+0) = -2K_X
\]
is ample, so by \cite[Thm 2.2]{Fuj11b} the restriction map
\[
H^0(X, -K_X) \rightarrow H^0(W, -K_X|_W)
\]
is surjective. Thus it is sufficient to show that
$H^0(W, -K_X|_W) \neq 0$.

Since $W$ is a minimal lc centre, there exists by Kawamata subadjunction an effective $\Q$-divisor $\Delta_W$ such that
$(W, \Delta_W)$ is klt and
\[
K_W + \Delta_W \sim_\Q K_X|_W.
\]
Thus $(W, \Delta_W)$ is log Fano.

Now we use our assumption: 
since $X$ is normal of dimension at most five,
the lc centre $W$ has dimension at most three. 
If $\dim W \leq 2$ we  have $H^0(W, -K_X|_W) \neq 0$ by \cite[Thm 3.1]{Kaw00}.
If $\dim W=3$, we have $H^0(W, -K_X|_W) \neq 0$ by Corollary \ref{corollarylogCY}.
\end{proof}

\subsection*{Acknowledgements} 
The first-named author was supported by the Institut Universitaire de France and the A.N.R. project
project Foliage (ANR-16-CE40-0008). 	
The second-named author was supported by Kartezjusz programme for PhD students at the University of Warsaw. Moreover he would like to thank his advisor, Jarosław Buczyński, for guidance and patience and Massimiliano Mella for the informative talk during his stay in Warsaw. Both authors would also like to thank anonymous referee for helpful remarks.

\bibliographystyle{alpha}
\bibliography{biblio}

\begin{thebibliography}{BCHM06}

\bibitem[Amb99]{Amb99}
Florin Ambro.
\newblock Ladders on {F}ano varieties.
\newblock {\em J. Math. Sci. (New York)}, 94(1):1126--1135, 1999.
\newblock Algebraic geometry, 9.

\bibitem[BCHM06]{BCHM06}
Caucher Birkar, Paolo Cascini, Christopher Hacon, and James McKernan.
\newblock Existence of minimal models for varieties of log general type.
\newblock {\em arXiv}, 2006.

\bibitem[Deb01]{Deb01}
Olivier Debarre.
\newblock {\em Higher-dimensional algebraic geometry}.
\newblock Universitext. Springer-Verlag, New York, 2001.

\bibitem[Flo13]{Flo13}
Enrica Floris.
\newblock Fundamental divisors on {F}ano varieties of index {$n-3$}.
\newblock {\em Geom. Dedicata}, 162:1--7, 2013.

\bibitem[Fuj75]{Fuj75}
Takao Fujita.
\newblock On the structure of polarized varieties with {$\Delta $}-genera zero.
\newblock {\em J. Fac. Sci. Univ. Tokyo Sect. IA Math.}, 22:103--115, 1975.

\bibitem[Fuj90]{Fuj90}
Takao Fujita.
\newblock {\em Classification theories of polarized varieties}, volume 155 of
  {\em London Mathematical Society Lecture Note Series}.
\newblock Cambridge University Press, Cambridge, 1990.

\bibitem[Fuj11]{Fuj11b}
Osamu Fujino.
\newblock Fundamental theorems for the log minimal model program.
\newblock {\em Publ. Res. Inst. Math. Sci.}, 47(3):727--789, 2011.

\bibitem[Heu16]{Heu16}
Liana Heuberger.
\newblock Deux points de vue sur les vari{\'e}t{\'e}s de {F}ano :
  g{\'e}om{\'e}trie du diviseur anticanonique et classification des surfaces
  {\`a} singularit{\'e}s 1/3(1,1).
\newblock {\em Ph.D thesis UPMC}, 2016.

\bibitem[H{\"o}r12]{Hoe12}
Andreas H{\"o}ring.
\newblock On a conjecture of {B}eltrametti and {S}ommese.
\newblock {\em J. Algebraic Geom.}, 21(4):721--751, 2012.

\bibitem[Hwa98]{Hw}
Jun-Muk Hwang.
\newblock Stability of tangent bundles of low-dimensional {F}ano manifolds with
  {P}icard number {$1$}.
\newblock {\em Math. Ann.}, 312(4):599--606, 1998.

\bibitem[IP99]{IP99}
V.~A. Iskovskikh and Yu.~G. Prokhorov.
\newblock Fano varieties.
\newblock In {\em Algebraic geometry, {V}}, volume~47 of {\em Encyclopaedia
  Math. Sci.}, pages 1--247. Springer, Berlin, 1999.

\bibitem[Kaw00]{Kaw00}
Yujiro Kawamata.
\newblock On effective non-vanishing and base-point-freeness.
\newblock {\em Asian J. Math.}, 4(1):173--181, 2000.
\newblock Kodaira's issue.

\bibitem[KM98]{KM98}
J{\'a}nos Koll{\'a}r and Shigefumi Mori.
\newblock {\em Birational geometry of algebraic varieties}, volume 134 of {\em
  Cambridge Tracts in Mathematics}.
\newblock Cambridge University Press, Cambridge, 1998.
\newblock With the collaboration of C. H. Clemens and A. Corti.

\bibitem[Kol13]{Kol13}
J{\'a}nos Koll{\'a}r.
\newblock {\em Singularities of the minimal model program}, volume 200 of {\em
  Cambridge Tracts in Mathematics}.
\newblock Cambridge University Press, Cambridge, 2013.
\newblock With a collaboration of S{\'a}ndor Kov{\'a}cs.

\bibitem[{Liu}17]{Liu17}
Jie {Liu}.
\newblock {Second Chern class of Fano manifolds and anti-canonical geometry}.
\newblock {\em arXiv e-prints}, page arXiv:1710.04116, Oct 2017.

\bibitem[Mel99]{Me}
Massimiliano Mella.
\newblock Existence of good divisors on {M}ukai varieties.
\newblock {\em J. Algebraic Geom.}, 8(2):197--206, 1999.

\bibitem[Ou17]{Ou17}
Wenhao Ou.
\newblock On generic nefness of tangent sheaves.
\newblock {\em arXiv preprint 1703.03175}, 2017.

\bibitem[Xie09]{Xie09}
Qihong Xie.
\newblock A note on the effective non-vanishing conjecture.
\newblock {\em Proc. Amer. Math. Soc.}, 137(1):61--63, 2009.

\end{thebibliography}
\end{document}